\documentclass[a4paper,11pt]{article}
\usepackage{amsmath, amsfonts, amscd, amssymb, amsthm, enumerate, xypic}

\def\bfB{\mathbf{B}}
\def\bfC{\mathbf{C}}

\newcommand{\Hom}{\operatorname{Hom}}
\newcommand{\Mat}{\operatorname{M}}
\newcommand{\Mats}{\operatorname{S}}
\newcommand{\Mata}{\operatorname{A}}

\newcommand{\Ker}{\operatorname{Ker}}

\newcommand{\Vect}{\operatorname{span}}
\newcommand{\im}{\operatorname{Im}}

\newcommand{\tr}{\operatorname{tr}}

\newcommand{\codim}{\operatorname{codim}}
\renewcommand{\setminus}{\smallsetminus}
\newcommand{\modu}{\operatorname{mod}}

\def\K{\mathbb{K}}

\def\calA{\mathcal{A}}
\def\calB{\mathcal{B}}

\def\calH{\mathcal{H}}

\def\calL{\mathcal{L}}
\def\calM{\mathcal{M}}
\def\calN{\mathcal{N}}

\def\calS{\mathcal{S}}
\def\calT{\mathcal{T}}
\def\calU{\mathcal{U}}

\def\lcro{\mathopen{[\![}}
\def\rcro{\mathclose{]\!]}}

\theoremstyle{definition}
\newtheorem{Def}{Definition}[section]
\newtheorem{Not}[Def]{Notation}

\theoremstyle{plain}
\newtheorem{theo}{Theorem}[section]
\newtheorem{prop}[theo]{Proposition}
\newtheorem{cor}[theo]{Corollary}
\newtheorem{lemma}[theo]{Lemma}
\newtheorem{claim}{Claim}
\newtheorem{step}{Step}

\theoremstyle{plain}

\theoremstyle{remark}
\newtheorem{Rems}{Remarks}
\newtheorem{Rem}[Rems]{Remark}
\newtheorem{ex}[Rems]{Example}

\title{Range-compatible homomorphisms on spaces of symmetric or alternating matrices}
\author{Cl\'ement de Seguins Pazzis\footnote{Universit\'e de Versailles Saint-Quentin-en-Yvelines, Laboratoire de Math\'ematiques
de Versailles, 45 avenue des Etats-Unis, 78035 Versailles cedex, France}
\footnote{e-mail address: dsp.prof@gmail.com}}

\begin{document}

% Relecture 4/4 accomplie.

\thispagestyle{plain}

\maketitle

\begin{abstract}
Let $U$ and $V$ be finite-dimensional vector spaces over an arbitrary field $\K$, and
$\calS$ be a linear subspace of the space $\calL(U,V)$ of all linear maps from $U$ to $V$.
A map $F : \calS \rightarrow V$ is called range-compatible when it satisfies
$F(s) \in \im s$ for all $s \in \calS$. Among the range-compatible maps are the so-called local ones, that is the maps of the form
$s \mapsto s(x)$ for a fixed vector $x$ of $U$.

In recent works, we have classified the range-compatible group homomorphisms on $\calS$ when
the codimension of $\calS$ in $\calL(U,V)$ is small.
In the present article, we study the special case when $\calS$ is a linear subspace of the space $\Mats_n(\K)$ of all $n$ by $n$ symmetric matrices:
we prove that if the codimension of $\calS$ in $\Mats_n(\K)$ is less than or equal to $n-2$, then every range-compatible homomorphism on $\calS$
is local provided that $\K$ does not have characteristic $2$. With the same assumption on the codimension of $\calS$, we also classify
the range-compatible homomorphisms on $\calS$ when $\K$ has characteristic $2$. Finally, we prove that
if $\calS$ is a linear subspace of the space $\Mata_n(\K)$ of all $n$ by $n$ alternating matrices with entries in $\K$,
and the codimension of $\calS$ is less than or equal to $n-3$, then every range-compatible homomorphism on $\calS$ is local.
\end{abstract}



\section{Introduction}

\subsection{Notation}

Throughout the article, all the vector spaces are assumed to be finite-dimensional
and over a given arbitrary field $\K$.
Given vector spaces $U$ and $V$, we denote by $\calL(U,V)$ the space of all linear maps from $U$ to $V$.
Given abelian groups $A$ and $B$, we denote by $\Hom(A,B)$ the group of all  homomorphisms from $A$ to $B$.
Group homomorphisms will be simply called homomorphisms in this article.

Given non-negative integers $n$ and $p$, we denote by $\Mat_{n,p}(\K)$ the space of all matrices with $n$ rows, $p$ columns and entries in $\K$.
For $(i,j)\in \lcro 1,n\rcro \times \lcro 1,p\rcro$, we denote by $E_{i,j}$ the matrix of $\Mat_{n,p}(\K)$
in which all the entries equal zero except the one at the $(i,j)$-spot, which equals $1$.
Given finite-dimensional vector spaces $U$ and $V$, equipped with respective bases $\bfB$ and $\bfC$,
and given a linear map $u : U \rightarrow V$, we denote by $\Mat_{\bfB,\bfC}(u)$ the matrix of $u$ in the bases $\bfB$ and $\bfC$.

We set $\Mat_n(\K):=\Mat_{n,n}(\K)$ and we denote by $\Mats_n(\K)$ the subspace of $\Mat_n(\K)$ consisting of its symmetric matrices,
and by $\Mata_n(\K)$ the subspace of all $n$ by $n$ alternating matrices with entries in $\K$.
Recall that a square matrix $M \in \Mat_n(\K)$ is called alternating whenever $X^T M X=0$ for all $X \in \K^n$, which is equivalent
to having $M$ skew-symmetric with all diagonal entries zero.

Given a square matrix $M=(m_{i,j}) \in \Mat_n(\K)$, we denote its diagonal vector by
$$\Delta(M):=\begin{bmatrix}
m_{1,1} \\
\vdots \\
m_{n,n}
\end{bmatrix}\in \K^n.$$
Given a vector $X=(x_i)_{1 \leq i \leq n} \in \K^n$ and a map $\alpha : \K \rightarrow \K$, we
denote by $X^\alpha:=(\alpha(x_i))_{1 \leq i \leq n}$ the vector of $\K^n$ deduced from $X$ by applying $\alpha$ entry-wise.

Let $m,n,p,q$ be non-negative integers. Given respective subsets $\calA$ and $\calB$ of $\Mat_{m,n}(\K)$ and $\Mat_{p,q}(\K)$, one
sets
$$\calA \oplus \calB:=\biggl\{\begin{bmatrix}
A & 0 \\
0 & B
\end{bmatrix} \mid A \in \calA, \; B \in \calB\biggr\},$$
which is a subset of $\Mat_{m+p,n+q}(\K)$.
Note that if $\calA$ and $\calB$ are respective linear subspaces of $\Mats_n(\K)$ and $\Mats_p(\K)$ (respectively, of $\Mata_n(\K)$ and $\Mata_p(\K)$),
then $\calA \oplus \calB$ is a linear subspace of $\Mats_{n+p}(\K)$ (respectively, of $\Mata_{n+p}(\K)$)
with dimension $\dim \calA+\dim \calB$.

\subsection{The problem}

Let $U$ and $V$ be vector spaces, and $\calS$ be a subset of $\calL(U,V)$.
A map $F : \calS \rightarrow V$ is called \textbf{range-compatible} whenever
$$\forall s \in \calS, \; F(s) \in \im s.$$
It is called \textbf{local} when there exists a vector $x \in U$ such that
$$\forall s \in \calS, \; F(s)=s(x).$$
Clearly, every local map is range-compatible, but the converse is not true, even for linear maps.

\vskip 3mm
The study of range-compatible maps is a very recent topic in linear algebra.
It is mainly motivated by its connection with spaces of bounded rank matrices and
from the fact that it is a reformulation of the concept of algebraic reflexivity.
The basic result in the theory of range-compatible maps is the following one.

\begin{theo}\label{totalspace}
Every range-compatible linear map on $\calL(U,V)$ is local.
\end{theo}

This result was independently discovered by several authors (it is implicit in \cite{Dieudonne}, for instance).
In \cite{dSPclass}, it was extended as follows.

\begin{theo}\label{maintheolin1}
Let $\calS$ be a linear subspace of $\calL(U,V)$ such that $\codim \calS \leq \dim V-2$.
Then, every range-compatible linear map on $\calS$ is local.
\end{theo}

Then, in \cite{dSPRC1}, this result was also extended to group homomorphisms, for which the upper-bound $\dim V-2$
turns out to be optimal\footnote{The upper-bound $\dim V-2$ is not optimal for linear maps, see Theorem 1.2 of \cite{dSPRC1}
for the optimal bound.}.

\begin{theo}\label{maintheogroup}
Let $\calS$ be a linear subspace of $\calL(U,V)$ such that $\codim \calS \leq \dim V-2$.
Then, every range-compatible homomorphism on $\calS$ is local.
\end{theo}

Theorem \ref{maintheolin1} is a major key in a generalization \cite{dSPclass} of an important result of Atkinson and Lloyd \cite{AtkLloyd} on the classification of large spaces of matrices with rank bounded above, and it was also used in a study of invertibility preservers on large
spaces of square matrices \cite{dSPlargelinpres}. Further work on range-compatibility has appeared recently: we refer to \cite{dSPRC1,dSPRC2,dSPRC3}
for the latest developments of this topic.

In this article, we shall focus on the special case when $\calS$ is a linear subspace of $\Mats_n(\K)$
or $\Mata_n(\K)$. Of course, $\calS$ is then naturally identified with a linear subspace of $\calL(\K^n,\K^n)$
through the canonical basis of $\K^n$. The main motivation for considering such subspaces is the potential application to
the study of large spaces of symmetric or alternating matrices with bounded rank: in further research, the results of this
article will be used to generalize theorems of Radwan and Loewy \cite{Loewy,LoewyRadwan} on the structure of such spaces.

The determination of the range-compatible homomorphisms on $\Mats_n(\K)$ was achieved in \cite{dSPRC1}.
For the characteristic $2$ case, we need some additional terminology before we can state the result.

\begin{Def}
Assume that $\K$ has characteristic $2$. Let $V$ and $V'$ be vector spaces over $\K$.
A \textbf{root-linear} map
$f : V \rightarrow V'$ is a group homomorphism such that
$$\forall (\lambda,x)\in \K \times V, \; f(\lambda^2 x)=\lambda f(x).$$
\end{Def}

Note that if $\K$ has characteristic $2$ and more than $2$ elements, then only the zero map from $V$ to $V'$ is both linear and root-linear.
However, if $\# \K=2$ root-linear maps coincide with linear maps.

\begin{theo}[Theorem 3.1 of \cite{dSPRC1}]\label{fullsymmetric}
Let $n \geq 2$ be an integer.
\begin{enumerate}[(a)]
\item If $\K$ has characteristic not $2$, then every range-compatible homomorphism on $\Mats_n(\K)$ is local.
\item If $\K$ has characteristic $2$, then the range-compatible homomorphisms on $\Mats_n(\K)$ are the sums of the local maps and
the maps of the form $M \mapsto \Delta(M)^\alpha$
where $\alpha$ is a root-linear map from $\K$ to $\K$.
\item If $\K$ has more than $2$ elements, then every range-compatible linear map on $\Mats_n(\K)$ is local.
\end{enumerate}
\end{theo}

\subsection{Main results}

\begin{Def}
Let $\calS$ be a linear subspace of $\Mats_n(\K)$. \\
A range-compatible homomorphism on $\calS$ is called \textbf{standard} when
it can be extended into a range-compatible homomorphism on the full space $\Mats_n(\K)$.
\end{Def}

Thus, if $\K$ has characteristic not $2$ and $n \geq 2$, then standard range-compatible homomorphisms
are just local maps. If $\K$ has characteristic $2$, then a range-compatible homomorphism on $\calS$
is standard if and only if it is the sum of a local map and of a map of the form $M \mapsto \Delta(M)^\alpha$
for some root-linear form $\alpha$ on $\K$.

Now, we can state our main result on range-compatible homomorphisms over spaces of symmetric matrices.

\begin{theo}\label{maintheosym}
Let $\calS$ be a linear subspace of $\Mats_n(\K)$ such that $\codim \calS \leq n-2$. \\
Then, every range-compatible homomorphism on $\calS$ is standard.
\end{theo}

If $\# \K>2$, only the zero map is both linear and root-linear. Thus, Theorem \ref{maintheosym} yields:

\begin{cor}
Let $\calS$ be a linear subspace of $\Mats_n(\K)$ such that $\codim \calS \leq n-2$ and $\# \K>2$.
Then, every range-compatible linear map on $\calS$ is local.
\end{cor}

The following counter-example demonstrates that the upper-bound $n-2$ is optimal.
For $n \geq 2$, consider the subspace $\calS:=\Mats_1(\K) \oplus \Mats_{n-1}(\K)$ of $\Mats_n(\K)$.
Note that $\calS$ has codimension $n-1$ in $\Mats_n(\K)$.
If there exists a non-linear homomorphism $\varphi : \K \rightarrow \K$ (which is the case if and only if $\K$ is not a prime field)
then the mapping
$$F : (m_{i,j}) \in \calS \mapsto \begin{bmatrix}
\varphi(m_{1,1}) \\
[0]_{(n-1) \times 1}
\end{bmatrix}$$
is obviously range-compatible but we check that it is not standard.
Indeed, as $F$ is non-linear, if it is standard then $\K$ has characteristic $2$ and more than $2$ elements
and there is a root-linear form $\alpha$ on $\K$ and linear forms $\beta_1,\dots,\beta_n$ on $\calS$ such that
$F : M=(m_{i,j}) \mapsto  \begin{bmatrix}
\beta_1(M) \\
\vdots \\
\beta_n(M)
\end{bmatrix}+\Delta(M)^\alpha$. Judging from the last entry in $F(M)$, we would find that $\alpha$ is both root-linear and linear, and hence $\alpha=0$;
 then we would see from the first entry of $F(M)$ that $\varphi$ is linear, contradicting our assumptions.

\begin{Rem}\label{symlinremark}
An open problem remains to find the optimal upper-bound on the codimension on $\calS$
for all range-compatible \emph{linear} maps on $\calS$ to be local, provided that $\K$ has more than $2$ elements of course.
We conjecture that a reasonable upper-bound should be $2\dim V-4$, except for some very small fields
for which special cases might arise. In any case, the following example shows that the upper-bound should be less than or equal to
$2 \dim V-4$.

Consider the space $\calU_2$ of all matrices of the form $\begin{bmatrix}
a & b \\
b & 0
\end{bmatrix}$ with $(a,b)\in \K^2$, and the map $F : (m_{i,j})\in \calU \mapsto \begin{bmatrix}
m_{1,1} \\
0
\end{bmatrix}$. The map $F$ is non-local. However, $F$ is range-compatible
because whenever $M \in \calU_2 \setminus \{0\}$ satisfies $F(M)\neq 0$,
either $F(M)$ is the first column of $M$ or else $M$ has rank $2$ and hence its range contains $F(M)$.
For all $n \geq 2$, we extend this counter-example as follows: we consider the space
$\calS:=\calU_2 \oplus \Mats_{n-2}(\K)$ and, for all $M=\begin{bmatrix}
A & [0] \\
[0] & B
\end{bmatrix}$ with $A \in \calU_2$ and $B \in \Mats_{n-2}(\K)$, we set
$$\widetilde{F}(M):=\begin{bmatrix}
F(A) \\
[0]_{(n-2) \times 1}
\end{bmatrix}.$$
Then, it is easy to check that $\widetilde{F}$ is linear, range-compatible and non-local, whereas $\codim \calS=2n-3$.
\end{Rem}

\vskip 3mm
Now, we turn to our main result on range-compatible homomorphisms on spaces of alternating matrices.

\begin{theo}\label{alttheo}
Let $\calS$ be a linear subspace of $\Mata_n(\K)$ such that $\codim \calS \leq n-3$. \\
Then, every range-compatible homomorphism on $\calS$ is local.
\end{theo}

The following example shows that the upper-bound $n-3$ is optimal.
Consider indeed the subspace $\calS$ consisting of all the alternating matrices $A=(a_{i,j})$
such that $a_{i,1}=0$ for all $i \in \lcro 3,n\rcro$.
If $\K$ is not a prime field, we can choose a non-linear group endomorphism $\varphi$ of $\K$,
and hence
$$A \in \calS \mapsto \begin{bmatrix}
0 \\
\varphi(a_{2,1}) \\
[0]_{(n-2) \times 1}
\end{bmatrix}$$
is obviously range-compatible, but it is non-linear and hence non-local.

As in the symmetric case, the upper bound $n-3$ in Theorem \ref{alttheo} does not seem to be optimal if
we restrict the study to range-compatible \emph{linear} maps.
The following examples suggest that the optimal upper-bound in this situation
could be $2n-6$ if $\# \K>2$, and $2n-7$ if $\# \K=2$,
unless $n \leq 3$ in which case no provision on the codimension of $\calS$ appears necessary.

\begin{ex}
Assume that $n \geq 4$.
Consider the space $\calU$ of all matrices of $\Mat_2(\K)$ of the form $\begin{bmatrix}
a & b \\
0 & a
\end{bmatrix}$ with $(a,b)\in \K^2$. As in Remark \ref{symlinremark}, one checks that $\begin{bmatrix}
a & b \\
0 & a
\end{bmatrix} \mapsto \begin{bmatrix}
b \\
0
\end{bmatrix}$ is range-compatible.
Now, consider the space $\calS$ of all matrices of the form
$$\begin{bmatrix}
[0]_{2 \times 2} & -A^T & [0]_{2 \times (n-4)} \\
A & B & -C^T \\
[0]_{(n-4) \times 2} & C & D
\end{bmatrix}$$
in which $A \in \calU$, $B \in \Mata_2(\K)$, $C \in \Mat_{n-4,2}(\K)$ and $D \in \Mata_{n-4}(\K)$.
One checks that $\calS$ is a linear subspace of $\Mata_n(\K)$ with codimension $2n-5$.
The linear map
$$F : (m_{i,j}) \in \calS \mapsto \begin{bmatrix}
0 \\
0 \\
m_{3,2} \\
[0]_{(n-3) \times 1}
\end{bmatrix}
$$
is range-compatible (as $F(M)$ is always a linear combination of the first two columns of $M$),
but one checks that it is non-local.
\end{ex}

\begin{ex}
Assume that $n \geq 4$ and $\# \K=2$.
We know from Theorem \ref{fullsymmetric} that
$\begin{bmatrix}
a & b \\
b & c
\end{bmatrix} \mapsto
\begin{bmatrix}
a \\
c
\end{bmatrix}$ is range-compatible.
Now, consider the space $\calS$ of all matrices of the form
$$\begin{bmatrix}
[0]_{2 \times 2} & -A^T & [0]_{2 \times (n-4)} \\
A & B & -C^T \\
[0]_{(n-4) \times 2} & C & D
\end{bmatrix}$$
in which $A \in \Mats_2(\K)$, $B \in \Mata_2(\K)$, $C \in \Mat_{n-4,2}(\K)$ and $D \in \Mata_{n-4}(\K)$.
One checks that $\calS$ is a linear subspace of $\Mata_n(\K)$ with codimension $2n-6$.
However, the linear map
$$F : (m_{i,j}) \in \calS \mapsto \begin{bmatrix}
0 \\
0 \\
m_{3,1} \\
m_{4,2} \\
[0]_{(n-4) \times 1}
\end{bmatrix}
$$
is range-compatible (as $F(M)$ is always a linear combination of the first two columns of $M$),
but one checks that it is non-local.
\end{ex}

As an easy consequence of Theorem \ref{alttheo}, we obtain the following result.

\begin{theo}\label{alttheofull}
Let $n$ be an arbitrary non-negative integer.
Every range-compatible linear map on $\Mata_n(\K)$ is local.
\end{theo}

Let us immediately show how this result follows from Theorem \ref{alttheo}.

\begin{proof}[Proof of Theorem \ref{alttheofull}]
For $n \geq 3$, the result is a direct consequence of Theorem \ref{alttheo}.
On the other hand, in a linear subspace $\calS=\K s_1$ of $\calL(U,V)$ with dimension at most $1$, every-range compatible linear map is local:
indeed, given a linear map $F : \calS \rightarrow V$ that is range-compatible, we have some $x \in U$ such that $F(s_1)=s_1(x)$,
and as the linear maps $F$ and $s \mapsto s(x)$ coincide on a generating subset of the vector space $\calS$, they are equal.
For $n \leq 2$, the claimed result thus follows from that general basic point.
\end{proof}

\subsection{Strategy and structure of the article}

In this article, we will follow the main method from \cite{dSPRC1}, which consists
in performing inductive proofs over the dimension of the target space $V$, by using quotient spaces.
Here, a major problem appears: if $\calS$ is a subspace of $\Mats_n(\K)$ or of $\Mata_n(\K)$, ``moding out"
a  non-zero subspace $V_0$ of $V$ yields an operator space $\calS \modu V_0$ which cannot be represented by a space
of symmetric of alternating matrices, simply because the source and target spaces of the operators in $\calS$
no longer have the same dimension! In order to rescue the quotient space technique, it is necessary to enlarge
the theorem as to encompass not only subspaces of $\Mats_n(\K)$, but also subspaces of $\Mats_n(\K) \coprod \Mat_{n,p}(\K)$
for arbitrary integers $n \geq 2$ and $p \geq 0$ (and ditto for alternating matrices instead of symmetric ones).
To avoid the need of constantly going from matrix spaces to operator spaces, we shall reframe
the results in terms of subspaces of self-adjoint or alternating operators from $U$ to the dual space of
one of its subspaces $V$. Then, the generalized versions of Theorems \ref{maintheosym} and \ref{alttheo}
can be proved by induction on the dimension of $V$ (with fixed $U$), and Theorems
\ref{maintheosym} and \ref{alttheo} follow from the special case when $U=V$.

The rest of the article is organized as follows. In Section \ref{techniques}, we give
a quick review of the main techniques for the study of range-compatible homomorphisms, that is the quotient space technique
and the splitting technique. We also introduce some notation on duality that is used throughout the text.
Sections \ref{RCsym} and \ref{RCalt} are devoted to the study of range-compatible homomorphisms, respectively on spaces of self-adjoint operators
and on spaces of alternating operators. These two sections are logically independent from one another, but they are essentially based
upon similar methods. In the self-adjoint case, the main difficulty comes from the case of fields with characteristic $2$,
whereas for alternating operators the proof is common to all fields, but substantially more technical.

\section{Quotient and splitting space techniques}\label{techniques}

\subsection{Quotient space techniques}

We recall the following notation and result from \cite{dSPRC1}.

\begin{Not}
Let $\calS$ be a linear subspace of $\calL(U,V)$ and $V_0$ be a linear subspace of $V$.
Denote by $\pi : V \rightarrow V/V_0$ the canonical projection.
Then, we set
$$\calS \modu V_0:=\bigl\{\pi \circ s \mid s \in \calS\bigr\},$$
which is a linear subspace of $\calL(U,V/V_0)$.
\end{Not}

\begin{lemma}\label{quotientgeneral}
Let $F : \calS \rightarrow V$ be a range-compatible homomorphism.
Then, there is a unique range-compatible homomorphism $F \modu V_0 : \calS \modu V_0 \rightarrow V/V_0$
such that
$$\forall s \in \calS, \; (F \modu V_0)(\pi \circ s)=\pi(F(s)),$$
i.e.\ such that the following diagram is commutative
$$\xymatrix{
\calS \ar[rr]^F \ar[d]_{s \mapsto \pi \circ s} & & V \ar[d]^\pi \\
\calS \modu V_0 \ar[rr]_{F \modu V_0} & & V/V_0.
}$$
\end{lemma}

When $V_0=\K y$ for some non-zero vector $y$, we simply write $\calS \modu y$ instead of $\calS \modu \K y$, and $F \modu y$ instead of
$F \modu \K y$.

In terms of matrices, the case when $V_0$ is a hyperplane reads as follows.

\begin{cor}\label{rowlemma}
Let $\calS$ be a linear subspace of $\Mat_{n,p}(\K)$ and $F : \calS \rightarrow \K^n$
be a range-compatible homomorphism. Then, there are group homomorphisms $F_1,\dots,F_n$
such that
$$\forall M \in \calS, \quad F(M)=\begin{bmatrix}
F_1(R_1(M)) \\
\vdots \\
F_n(R_n(M))
\end{bmatrix}$$
where $R_1(M),\dots,R_n(M)$ denote the rows of $M$.
\end{cor}

\subsection{On duality}

\begin{Not}
Given a vector space $V$, the dual space of $V$ (i.e.\ the set of all linear forms on $V$) is denoted by $V^\star$.
\end{Not}

\begin{Def}
Let $V$ be a linear subspace of the vector space $U$.
We denote by $V^o$ the linear subspace of $U^\star$ consisting of the linear forms on $U$ that vanish everywhere on $V$.

Let $W$ be a linear subspace of $U^\star$. We denote by ${}^o W$ the linear subspace of $U$ consisting of the vectors
that annihilate all the linear forms in $W$.
\end{Def}

In any case, the letter ``$o$" stands for ``orthogonal".

Remember that if $U$ is finite-dimensional then ${}^o(V^o)=V$ for every linear subspace $V$ of $U$, and
$({}^o W)^o=W$ for every linear subspace $W$ of $U^\star$.
Moreover, for every linear subspace $W$ of $U^\star$, we have a natural isomorphism
$$U^\star/W \overset{\simeq}{\longrightarrow} ({}^o W)^\star.$$
In the rest of the article, we shall systematically use this isomorphism to identify  $U^\star/W$ with $({}^o W)^\star$.
In particular, given a linear form $f^\star \in U^\star$, the space $U^\star/\K f^\star$ is naturally identified with $(\Ker f^\star)^\star$.

We finish with some terminology on dual bases.

\begin{Def}
Let $U$ be a finite-dimensional vector space and $V$ be a linear subspace of $U$, with respective dimensions $p$ and $n$.
A \textbf{compatible pair of bases} of $U$ and $V^\star$ is a pair $(\bfB_1,\bfB_2^\star)$
in which $\bfB_1=(e_1,\dots,e_p)$ is a basis of $U$ such that $(e_1,\dots,e_n)$ is a basis of $V$ whose dual basis is $\bfB_2^\star$.
\end{Def}

\subsection{The splitting technique}

\begin{Not}
Let $n,p,q$ be non-negative integers, and $\calA$ and $\calB$ be respective subsets of $\Mat_{n,p}(\K)$ and $\Mat_{n,q}(\K)$.
We set
$$\calA \coprod \calB :=\Bigl\{\begin{bmatrix}
A & B
\end{bmatrix} \mid (A,B) \in \calA \times \calB\Bigr\}.$$
\end{Not}

The following easy lemma was established in \cite{dSPRC1}.

\begin{lemma}[Splitting Lemma]\label{splittinglemma}
Let $n,p,q$ be non-negative integers, and $\calA$ and $\calB$ be linear subspaces, respectively, of $\Mat_{n,p}(\K)$ and $\Mat_{n,q}(\K)$.

Given maps $f : \calA \rightarrow \K^n$ and $g : \calB \rightarrow \K^n$,
set
$$f \coprod g : \begin{bmatrix}
A & B
\end{bmatrix} \in \calA \coprod \calB \longmapsto f(A)+g(B).$$
Then:
\begin{enumerate}[(a)]
\item The homomorphisms from $\calA \coprod \calB$ to $\K^n$
are the maps of the form $f \coprod g$, where $f \in \Hom(\calA,\K^n)$ and $g \in \Hom(\calB,\K^n)$.
\item Given $f \in \Hom(\calA,\K^n)$ and $g \in \Hom(\calB,\K^n)$, the homomorphism
$f \coprod g$ is range-compatible (respectively, local) if and only if $f$ and $g$ are range-compatible (respectively, local).
\end{enumerate}
\end{lemma}

\section{Range-compatible homomorphisms on large spaces of symmetric matrices}\label{RCsym}

\subsection{Spaces of self-adjoint operators}

The main obstacle in trying to prove Theorem \ref{maintheosym} by induction is that, given a vector $y \in \K^n \setminus \{0\}$,
the operator space $\calS \modu y$ can no longer be represented by a space of symmetric matrices.
In order to circumvent that problem, we shall both broaden the result and frame it in terms of operators rather than in terms of matrices.

\begin{Def}
Let $U$ be a vector space and $V$ be a linear subspace of $U$.
A linear map $f : U \rightarrow V^\star$ is called \textbf{self-adjoint} when
$$\forall (x,y)\in V^2, \quad f(x)[y]=f(y)[x],$$
i.e.\ when the bilinear form $(x,y) \in V^2 \mapsto f(x)[y]$ is symmetric.
The set of all self-adjoint linear maps from $U$ to $V^\star$ is denoted by $\calL_s(U,V^\star)$.
One checks that it is a linear subspace of $\calL(U,V^\star)$.
\end{Def}

Let $(\bfB_1,\bfB_2^\star)$ be a compatible pair of bases of $U$ and $U^\star$.
A linear map from $U$ to $U^\star$ is self-adjoint if and only if its matrix in $\bfB_1$ and $\bfB_2^\star$
is symmetric. Hence, we have a vector space isomorphism $f \in \calL_s(U,U^\star) \longmapsto M_{\bfB_1,\bfB_2^\star}(f) \in \Mats_n(\K)$,
where $n:=\dim U$.
More generally, given a compatible pair $(\bfB_1,\bfB_2^\star)$ of bases of $U$ and $V^\star$,
a linear map $f : U \rightarrow V^\star$ is self-adjoint if and only if its matrix in the bases $\bfB_1$
and $\bfB_2^\star$ belongs to $\Mats_n(\K) \coprod \Mat_{n,p-n}(\K)$, where $n:=\dim V$ and $p:=\dim U$.

\begin{Rem}\label{Rem4}
Let $W$ be a linear subspace of $V^\star$.
Then, through the canonical identification between $V^\star/W$ and $({}^o W)^\star$,
the space $\calL_s(U,V^\star) \modu W$ corresponds to $\calL_s(U,({}^o W)^\star)$.

In particular, $\calL_s(U,V^\star)$ corresponds to $\calL_s(U,U^\star) \modu V^o$.

Note these identifications have no impact on the problems under scrutiny, which follows from the
following general observation: given a linear subspace $\calS$ of $\calL(U_1,U_2)$ (where $U_1$ and $U_2$ are arbitrary vector spaces)
, a vector space isomorphism $\varphi : U_2 \overset{\simeq}{\rightarrow} U'_2$
and a group homomorphism $F : \calS \rightarrow U_2$, the linear subspace $\calS':=\{\varphi \circ s \mid s \in \calS\}$ of $\calL(U_1,U'_2)$
is isomorphic to $\calS$ through $\Phi : s \mapsto \varphi \circ s$, and $F':=\varphi \circ F \circ \Phi^{-1}$ is a group homomorphism from $\calS'$ to $U'_2$; one checks that $F$ is range-compatible (respectively, local) if and only if $F'$ is range-compatible (respectively, local).
\end{Rem}

Combining the splitting lemma and Theorems \ref{maintheolin1} and \ref{fullsymmetric},
we can easily describe the general form of range-compatible homomorphisms on $\Mats_n(\K) \coprod \Mat_{n,p}(\K)$
for arbitrary integers $n \geq 2$ and $p \geq 0$. The interpretation in terms of operators then reads as follows.

\begin{prop}\label{operatorsymprop}
Let $U$ be a finite-dimensional vector space and $V$ be a linear subspace of $U$ with $\dim V \geq 2$.
Let $\bigl((e_1,\dots,e_p),(e_1^\star,\dots,e_n^\star)\bigr)$ be a compatible pair of bases of $U$ and $V^\star$.
\begin{enumerate}[(a)]
\item If $\K$ does not have characteristic $2$ then every range-compatible homomorphism on $\calL_s(U,V^\star)$ is local.
\item If $\K$ has characteristic $2$ then every range-compatible homomorphism on $\calL_s(U,V^\star)$ splits in a unique fashion as the sum
of a local map and of $s \mapsto \underset{k=1}{\overset{n}{\sum}} \alpha(s(e_k)[e_k])\,e_k^\star$
for some root-linear form $\alpha$ on $\K$.
\end{enumerate}
\end{prop}

As a corollary, we obtain the following result.

\begin{cor}
Let $U$ be a finite-dimensional vector space and $V$ be a linear subspace of $U$ with $\dim V \geq 2$.
Then, every range-compatible homomorphism on $\calL_s(U,V^\star)$ can be written as
$G \modu V^o$ for some range-compatible homomorphism $G$ on $\calL_s(U,U^\star)$.
\end{cor}

\begin{proof}
Let $F$ be a range-compatible homomorphism on $\calL_s(U,V^\star)$.
If $F$ is local, we have $F : s \mapsto s(x)$ for some $x \in U$ and then it is clear that $F=G \modu V^o$
for $G : s \in \calL_s(U,U^\star) \mapsto s(x)$.

Now, assume that $\K$ has characteristic $2$.
Let $(e_1^\star,\dots,e_n^\star)$ be a basis of $V^\star$, with pre-dual basis $(e_1,\dots,e_n)$ of $V$,
which we extend into a basis $(e_1,\dots,e_p)$ of $U$.
Denote by $(f_1^\star,\dots,f_p^\star)$ the dual basis of $(e_1,\dots,e_p)$.
One sees that $f_{n+1}^\star,\dots,f_p^\star$ belong to $V^o$, whereas $(f_i^\star)_{|V}=e_i^\star$ for all $i \in \lcro 1,n\rcro$.

By Proposition \ref{operatorsymprop},
the map $F$ splits as $F_1+F_2$ with some local map $F_1 : s \mapsto s(x)$
and with $F_2 : s \mapsto \underset{k=1}{\overset{n}{\sum}} \alpha(s(e_k)[e_k])\, e_k^\star$ for some root-linear form $\alpha$ on $\K$.
Setting $G_2 : s \in \calL_s(U,U^\star) \mapsto \underset{k=1}{\overset{p}{\sum}} \alpha(s(e_k)[e_k])\,f_k^\star$,
we check that $G_2 \modu V^o=F_2$, whereas $G_1 \modu V^o=F_1$ for $G_1 : s \in \calL_s(U,U^\star) \mapsto s(x)$.
Hence, $F=(G_1+G_2) \modu V^o$ and $G_1+G_2$ is range-compatible.
\end{proof}

Now, we can extend the definition of standard maps as follows.

\begin{Def}
Let $\calS$ be a linear subspace of $\calL_s(U,V^\star)$.
A range-compatible homomorphism on $\calS$ is called \textbf{standard} when
it can be extended into a range-compatible homomorphism on the full space $\calL_s(U,V^\star)$.
\end{Def}

In particular, when $\K$ has characteristic not $2$ the standard range-compatible homomorphisms are simply the local ones.
Note that in any case the set of all standard range-compatible homomorphisms on $\calS$ is a linear subspace of $\Hom(\calS,V^\star)$.

We are now ready to state our main result on range-compatible homomorphisms on spaces of self-adjoint operators.

\begin{theo}\label{generalizedsymtheo}
Let $U$ be a finite-dimensional vector space, $V$ be a linear subspace of $U$, and $\calS$
be a linear subspace of $\calL_s(U,V^\star)$. Assume that $\codim \calS \leq \dim V-2$ (here, the codimension of $\calS$
is meant to be the one of $\calS$ in the space $\calL_s(U,V^\star)$, not in the full space $\calL(U,V^\star)$).
Then, every range-compatible homomorphism on $\calS$ is standard.
\end{theo}

The special case $U=V$ is simply Theorem \ref{maintheosym}.

The proof will require the following additional notation and remarks.
Let $\calS$ be a linear subspace of $\calL_s(U,V^\star)$.
To $\calS$, we can attach two operator spaces:
\begin{itemize}
\item The space
$$\calS_r:=\bigl\{f_{|V} \mid f \in \calS\bigr\},$$
which is a linear subspace of $\calL_s(V,V^\star)$ (the subscript ``r" stands for ``restricted");
\item And the space of all $f \in \calS$ that vanish everywhere on $V$, which we denote by $\calS_m$
and naturally reinterpret as a linear subspace of $\calL(U/V, V^\star)$ (the subscript ``m" stands for ``modulo").
\end{itemize}
By the rank theorem, we have
$$\dim \calS=\dim \calS_r+\dim \calS_m,$$
and on the other hand
$$\dim \calL_s(U,V^\star)=\dim \calL_s(V,V^\star)+\dim \calL(U/V,V^\star).$$

Let $(\bfB_1,\bfB_2^\star)$ be a compatible pair of bases of $U$ and $V^\star$, and denote by $\calM$ the matrix space representing
$\calS$ in them. Write $\bfB_1=(e_1,\dots,e_p)$ and set $\bfB'_1:=(e_1,\dots,e_n)$ and $\bfB''_1:=(\overline{e_{n+1}},\dots,\overline{e_p})$,
which are bases of $V$ and $U/V$, respectively.
Every $M \in \calM$ splits as
$$M=\begin{bmatrix}
S(M) & R(M)
\end{bmatrix} \quad \text{with $S(M) \in \Mats_n(\K)$ and $R(M) \in \Mat_{n,p-n}(\K)$.}$$
Then, the matrix space $S(\calM)$ represents $\calS_r$ in the bases $\bfB'_1$ and $\bfB_1^\star$.
On the other hand, if we denote by $\calT$ the subspace of $\calM$ consisting of all its matrices $M$ such that $S(M)=0$,
then $R(\calT)$ represents $\calS_m$ in the bases $\bfB''_1$ and $\bfB_1^\star$.

\subsection{Basic lemmas}

\begin{lemma}\label{rank1lemma}
Let $U$ be a finite-dimensional vector space with $\dim U \geq 3$.
Let $W$ be a proper linear subspace of $\calL_s(U,U^\star)$.
Then, there exist non-collinear vectors $f_1^\star$ and $f_2^\star$ of $U^\star$
such that $W$ contains no operator with range $\K f_1^\star$ and no operator with range $\K f_2^\star$.
\end{lemma}

\begin{proof}
Assume on the contrary that there exists a non-zero linear form $f^\star \in U^\star$ such that
$W$ contains a rank $1$ operator with range $g^\star$ for all $g^\star \in U^\star \setminus \K f^\star$.
We can find a basis $\bfB$ of $U$ such that the coordinates of $f^\star$ in the dual basis $\bfB^\star$
are all equal to $1$. Denote by $\calM$ the matrix space that represents $W$ in the bases $\bfB$ and $\bfB^\star$.
Then, $\calM$ should be a proper linear subspace of $\Mats_n(\K)$ where $n:=\dim U$.
Set $X_0:=\begin{bmatrix}
1 & \cdots & 1
\end{bmatrix}^T \in \K^n$. For all $X \in \K^n \setminus \K X_0$,
we know that $\calM$ contains $a XX^T$ for some $a \in \K \setminus \{0\}$, and hence $\calM$ actually contains $XX^T$
for all such $X$.
In particular, as $n \geq 3$, we see that $\calM$ contains $E_{i,i}$ for all $i \in \lcro 1,n\rcro$, and it also contains
$E_{i,i}+E_{j,j}+E_{i,j}+E_{j,i}$ for all $i$ and $j$ in $\lcro 1,n\rcro^2$ with $i \neq j$.
Hence, it contains $E_{i,j}+E_{j,i}$ for all $i$ and $j$ in $\lcro 1,n\rcro^2$ with $i \neq j$.
By linearly combining those special matrices, we obtain that $\calM=\Mats_n(\K)$, contradicting the fact that
$\calM$ is a proper subspace of $\Mats_n(\K)$.

Hence, for all $f^\star \in U^\star$, there exists $g^\star \in U^\star \setminus \K f^\star$
such that $W$ contains no operator with range $\K g^\star$.
Starting from $f^\star=0$, we obtain $f_1^\star \in U^\star \setminus \{0\}$ such that $W$ contains no operator with range $\K f_1^\star$,
and then we apply the result once more to find some $f_2^\star \in U^\star \setminus \K f_1^\star$ such that
$W$ contains no operator with range $\K f_2^\star$.
\end{proof}

\begin{lemma}\label{existtwovectorssym}
Let $U$ be a vector space, $V$ be a linear subspace of $U$, and $\calS$
be a linear subspace of $\calL_s(U,V^\star)$. Assume that $\codim \calS \leq \dim V-2$ and $\dim V \geq 3$.
Then, there are non-collinear vectors $f^\star$ and $g^\star$ in $V^\star$ such that
$\codim (\calS \modu f^\star) \leq \dim V-3$ and $\codim (\calS \modu g^\star) \leq \dim V-3$.
\end{lemma}

\begin{proof}
Let $f^\star$ be a non-zero vector of $V^\star$. We say that $f^\star$ is $\calS$-good
if $\codim (\calS \modu f^\star) \leq \dim V-3$. Thus, $f^\star$ is not $\calS$-good if and only if
$\codim \calS=\dim V-2$ and $\calS$ contains every rank $1$ operator of $\calL_s(U,V^\star)$ with range $\K f^\star$.
Assume that we cannot find non-collinear $\calS$-good vectors. Then, we can find a basis $(f_1^\star,\cdots,f_n^\star)$ of $V^\star$
in which no vector is $\calS$-good. In particular, for all $i \in \lcro 1,n\rcro$,
the space $\calS$ contains every operator whose kernel includes $V$ and whose range is included in $\K f_i^\star$.
Summing such operators yields that $\calS_m=\calL(U/V,V^\star)$, whence
$$\codim \calS_r=\codim \calS=\dim V-2>0.$$
Applying Lemma \ref{rank1lemma}, we deduce that there are non-collinear vectors
$f^\star$ and $g^\star$ of $V^\star$ such that $\calS_r$ contains no operator whose range equals either $\K f^\star$
or $\K g^\star$. Then, $f^\star$ and $g^\star$ are both $\calS$-good, which completes the proof.
\end{proof}

\begin{lemma}\label{twovectorssymlemma}
Let $\calS$ be a linear subspace of $\calL_s(U,V^\star)$ such that $\codim \calS \leq \dim V-2$ and $\dim V \geq 3$.
Let $F : \calS \rightarrow V^\star$ be a range-compatible homomorphism.
Assume that there exist non-collinear vectors $f_1^\star$ and $f_2^\star$ in $V^\star$
such that both maps $F \modu f_1^\star$ and $F \modu f_2^\star$ are local.
Then, $F$ is local.
\end{lemma}

\begin{proof}
Subtracting a local map from $F$, we lose no generality in assuming that
$F \modu f_1^\star=0$, in which case we have a vector $x \in U$ such that
$$\forall s \in \calS, \quad F(s)=s(x) \mod \K f_2^\star,$$
and
$$\forall s \in \calS, \; F(s) \in \K f_1^\star.$$
In particular, we note that
$$\forall s \in \calS, \quad s(x) \in \Vect(f_1^\star,f_2^\star).$$

Now, we show that $x \neq 0$ would lead to a contradiction.
We distinguish between three cases.

\vskip 3mm
\noindent \textbf{Case 1:} $x \not\in V$. \\
We choose a basis $\calB=(e_1,\dots,e_n)$ of $V$, whose dual basis we denote by $\bfB^\star$.
Then, $(e_1,\dots,e_n,x)$ can be extended into a basis $\widetilde{\bfB}$ of $U$.
We denote by $\calM$ the space of matrices representing $\calS$ in $\widetilde{\bfB}$ and $\bfB^\star$, and we
denote by $P$ the ($2$-dimensional) space of coordinates of the vectors of $\Vect(f_1^\star,f_2^\star)$ in $\bfB^\star$.
It follows from the above that $\calM$ is a subspace of $\calM':=\Mats_n(\K) \coprod P \coprod \Mat_{n,p-n-1}(\K)$.
However, $\dim \calM \geq \dim \calM'$, whence $\calM=\calM'$.
Denote by $G : \calM \rightarrow \K^n$ the map on matrices that corresponds to $F$ in the above bases.
Then, the range-compatible homomorphism $G' : P \rightarrow \K^n$ deduced from restricting
$G$ to $\{0\} \coprod P \coprod \{0\}$
has rank $1$! This is absurd as $G'$, being local by Theorem \ref{maintheogroup}, should be a scalar multiple of the identity of $P$.

\vskip 3mm
\noindent \textbf{Case 2:} $x \in {}^o \Vect(f_1^\star,f_2^\star)$ and $x \neq 0$. \\
In that case, we choose two vectors $e_1,e_2$ of $V$ such that $f_i^\star(e_j)=\delta_{i,j}$ for all $(i,j)\in \lcro 1,2\rcro^2$,
we note that $x \not\in \Vect(e_1,e_2)$, and we extend $(e_1,e_2)$ into a basis $\bfB=(e_1,\dots,e_n)$ of $V$ such that $e_n=x$ and all the vectors
$e_3,\dots,e_n$ belong to ${}^o \Vect(f_1^\star,f_2^\star)$. Then, the dual basis of $V^\star$ reads $(f_1^\star,f_2^\star,\dots,f_n^\star)$.
Finally, we extend $\bfB$ into a basis $\widetilde{\bfB}$ of $U$.
In the bases $\widetilde{\bfB}$ and $\bfB^\star$, the matrix space $\calM$ that corresponds to $\calS$
is included in $\calN \coprod \Mat_{n,p-n}(\K)$, where $\calN$ denotes the subspace of $\Mats_n(\K)$
consisting of the matrices $(s_{i,j})$ such that $s_{i,n}=0$ for all $i \in \lcro 3,n\rcro$.
Obviously, $\calN \coprod \Mat_{n,p-n}(\K)$ has codimension $n-2$ in $\Mats_n(\K) \coprod \Mat_{n,p-n}(\K)$,
whence $\calM=\calN \coprod \Mat_{n,p-n}(\K)$. Moreover, the map that corresponds to $F$ in those bases is
$$G : (m_{i,j}) \in \calM \longmapsto \begin{bmatrix}
m_{1,n} \\
[0]_{(n-1) \times 1}
\end{bmatrix}.$$
Define $M=(m_{i,j}) \in \calM$ by $m_{i,j}=1$ whenever $i \leq 2$ or $j \leq 2$, and
$m_{i,j}=0$ otherwise. Then, $G(M)=\begin{bmatrix}
1 \\
[0]_{(n-1) \times 1}
\end{bmatrix}$, but one easily checks that the range of $M$ is spanned by $\begin{bmatrix}
1 & 1 & 0 & \cdots & 0 \end{bmatrix}^T$
and $\begin{bmatrix}
0 & 0 & 1 & \cdots & 1
\end{bmatrix}^T$, and hence it is obvious that this range does not contain $G(M)$.

\vskip 3mm
\noindent \textbf{Case 3:} $x \in V \setminus \,{}^o\Vect(f_1^\star,f_2^\star)$. \\
We can extend $(f_1^\star,f_2^\star)$ into a basis $\bfB^\star=(f_1^\star,\dots,f_n^\star)$ of $V^\star$
such that $f_i^\star(x)=0$ for all $i \in \lcro 3,n\rcro$.
We denote by $(e_1,\dots,e_n)$ its dual basis of $V$, which we extend into a basis $\widetilde{\bfB}=(e_1,\dots,e_p)$ of $U$.
Note that $x \in \Vect(e_1,e_2)$, and hence there are scalars $\alpha$ and $\beta$ such that $x=\alpha\,e_1+\beta\,e_2$.
Now, denote by $\calM$ the space of matrices representing $\calS$ in $\widetilde{\bfB}$ and $\bfB^\star$.
Then, $\calM$ is included in $\calN_1 \coprod \Mat_{n,p-n}(\K)$, where $\calN_1$ denotes the subspace of $\Mats_n(\K)$
consisting of all the matrices $(s_{i,j})$ such that $\alpha\, s_{i,1}+\beta\,s_{i,2}=0$ for all $i \in \lcro 3,n\rcro$.

As $\codim \calS \leq n-2$, we see that $\dim (\calN_1 \coprod \Mat_{n,p-n}(\K)) \leq \dim \calM$, and hence
$\calM=\calN_1 \coprod \Mat_{n,p-n}(\K)$. Then, the map that corresponds to $F$
through the above bases reads
$$G : (s_{i,j}) \mapsto \begin{bmatrix}
\alpha\,s_{1,1}+\beta\,s_{1,2} \\
[0]_{(n-1) \times 1}
\end{bmatrix},$$
and hence it is range-compatible.
It follows from the splitting lemma that
$$G' : (s_{i,j})\in \calN_1 \mapsto \begin{bmatrix}
\alpha\,s_{1,1}+\beta\,s_{1,2} \\
[0]_{(n-1) \times 1}
\end{bmatrix}$$
is also range-compatible.

Applying $G'$ to any matrix of the form $\begin{bmatrix}
S & [0]_{2 \times (n-2)} \\
[0]_{(n-2) \times 2} & [0]_{(n-2) \times (n-2)}
\end{bmatrix}$, with $S \in \Mats_2(\K)$, yields that the linear map
$$\begin{bmatrix}
a & b \\
b & c
\end{bmatrix} \in \Mats_2(\K) \mapsto \begin{bmatrix}
\alpha a+\beta b \\
0
\end{bmatrix}$$
is range-compatible. Yet, since $(\alpha,\beta) \neq (0,0)$ it is obvious that this map is non-local.
Hence, by Theorem \ref{fullsymmetric} we have $\# \K=2$. Applying the above map to $\begin{bmatrix}
1 & 1 \\
1 & 1
\end{bmatrix}$ then yields $\alpha+\beta=0$, whence $\alpha=\beta=1$ (since $(\alpha,\beta) \neq (0,0)$).
A final contradiction then comes by applying $G'$ to the matrix $M=(m_{i,j}) \in \calN_1$ defined by
$m_{1,j}=m_{2,j}=m_{j,1}=m_{j,2}=1$ for all $j \in \lcro 3,n\rcro$, $m_{1,2}=1$, $m_{1,1}=m_{2,2}=0$ and $m_{i,j}=0$ for all
$(i,j)\in \lcro 3,n\rcro^2$. One checks that the range of $M$ is spanned by the vectors
$\begin{bmatrix}
0 \\
1 \\
1 \\
\vdots \\
1
\end{bmatrix}$ and $\begin{bmatrix}
1 \\
0 \\
1 \\
\vdots \\
1
\end{bmatrix}$ and that it does not include $\begin{bmatrix}
1 \\
[0]_{(n-1) \times 1}
\end{bmatrix}=G'(M)$. This contradicts the fact that $G'$ is range-compatible.

\vskip 3mm
In any of the above cases, we have found a contradiction. The only remaining possibility is that $x=0$. Then, $F$ is the local map $s \mapsto s(0)$.
\end{proof}

\subsection{Proof of Theorem \ref{generalizedsymtheo} for fields with characteristic not $2$}

Here, we assume that the characteristic of $\K$ is not $2$.

We prove Theorem \ref{generalizedsymtheo} by induction on the dimension of $V$.
The result is vacuous if $\dim V \leq 1$.
If $\dim V=2$, then the result is obvious because the assumptions tell us that $\calS=\calL_s(U,V^\star)$.
Assume now that $\dim V \geq 3$. By Lemma \ref{existtwovectorssym}, we can pick non-collinear vectors $f_1^\star$ and $f_2^\star$
in $V^\star$ such that $\codim (\calS \modu f_1^\star) \leq \dim V-3$ and $\codim (\calS \modu f_2^\star) \leq \dim V-3$.
However, we have a canonical identification of $\calS \modu f_i^\star$ with a linear subspace of
$\calL_s(U,(\Ker f_i^\star)^\star)$, and $\dim \Ker f_i^\star=\dim V-1$.
Thus, by induction, each map $F \modu f_i^\star$ is local. Lemma \ref{twovectorssymlemma}
then yields that $F$ is local. This completes the proof.

\subsection{Proof of Theorem \ref{generalizedsymtheo} for fields with characteristic $2$ and more than $2$ elements}

Here, we assume that $\K$ has characteristic $2$ and more than $2$ elements.
Again, we prove the result by induction on the dimension of $V$. The case when $\dim V \leq 2$
is either vacuous or trivial. Assuming that $\dim V \geq 3$ and considering an arbitrary range-compatible homomorphism
$F : \calS \rightarrow V^\star$, we proceed as in the above proof
to find linearly independent forms $f_1^\star$ and $f_2^\star$ in $V^\star$
such that both maps $F \modu f_1^\star$ and $F \modu f_2^\star$ are standard.
Then, $F \modu f_1^\star$ equals $F' \modu f_1^\star$ for some standard map $F'$ on $\calS$.
Thus, replacing $F$ with $F-F'$, we see that no generality is lost in assuming that
$$\forall s \in \calS, \quad F(s) \in \K f_1^\star.$$

\begin{claim}
The map $F \modu f_2^\star$ is local.
\end{claim}

\begin{proof}
Assume on the contrary that $F \modu f_2^\star$ is non-local.
Then, $F \modu f_2^\star$ is the sum of a local map $x \mapsto s(x)$, for some vector $x \in U$,
and of a non-zero root-linear map.
Let us extend $(f_1^\star,f_2^\star)$ into a basis $(f_1^\star,f_2^\star,\dots)$ of $V^\star$,
and then consider the dual basis $(e_1,\dots,e_n)$ of $V$, which we extend into a basis $(e_1,\dots,e_p)$ of $U$.
Denote by $\calM$ the matrix space associated with $\calS$ in the bases $(e_1,\dots,e_p)$ and $(f_1^\star,\dots,f_n^\star)$,
and by $G : \calM \rightarrow \K^n$ the map on matrices that corresponds to $F$.
One deduces that there is a non-zero root-linear form $\alpha$ on $\K$ together with linear form $\beta_1,\beta_3,\dots,\beta_n$ on $\calM$ such that
$$G : M=(m_{i,j}) \longmapsto \begin{bmatrix}
\alpha(m_{1,1})+\beta_1(M) \\
? \\
\alpha(m_{3,3})+\beta_3(M) \\
\vdots \\
\alpha(m_{n,n})+\beta_{n}(M)
\end{bmatrix}$$
On the other hand, as $F \modu f_1^\star=0$, it turns out that
$$\forall i \in \lcro 3,n\rcro, \; \forall M \in \calS, \quad \alpha(m_{i,i})=\beta_i(M).$$
As $\K$ has more than $2$ elements, only the zero map is both root-linear and linear, and hence $\beta_3=\cdots=\beta_n=0$
and, as $\alpha$ is non-zero, we deduce that
$$\forall i \in \lcro 3,n\rcro, \; \forall M \in \calS, \; m_{i,i}=0.$$
Denote by $\calT$ the subspace of $\Mats_n(\K)$ consisting of all the matrices $M=(m_{i,j}) \in \Mats_n(\K)$ such that
$\forall i \in \lcro 3,n\rcro, \; m_{i,i}=0$. Then, $\calM \subset \calT \coprod \Mat_{n,p-n}(\K)$
and $\dim (\calT \coprod \Mat_{n,p-n}(\K)) \leq \dim \calS$, whence $\calM = \calT \coprod \Mat_{n,p-n}(\K)$.

Denote by $H : \Mats_2(\K) \rightarrow \K^2$ the homomorphism such that
$$\forall A \in \Mats_2(\K), \; G\left(\begin{bmatrix}
A & [0]_{2 \times (p-2)} \\
[0]_{(n-2) \times 2} & [0]_{(n-2) \times (p-2)}
\end{bmatrix}\right) = \begin{bmatrix}
H(A) \\
[0]_{(n-2) \times 1}
\end{bmatrix}.$$
As $G$ is range-compatible, so is $H$.
Note that we have a linear form $\theta$ on $\Mats_2(\K)$ such that
$$H : A=(a_{i,j}) \mapsto \begin{bmatrix}
\alpha(a_{1,1})+\theta(A) \\
0
\end{bmatrix}.$$

Then, by Theorem \ref{fullsymmetric}, the map $H$ equals
$$A=(a_{i,j}) \mapsto \begin{bmatrix}
\alpha'(a_{1,1})+\theta_1(A) \\
\alpha'(a_{2,2})+\theta_2(A)
\end{bmatrix}$$
for some root-linear form $\alpha'$ on $\K$ and some linear forms $\theta_1$ and $\theta_2$ on $\Mats_2(\K)$.
As above we would deduce from looking at the first row that
$\alpha'=\alpha$, and from looking at the second row that $\alpha'=0$.
This leads to $\alpha=0$, contradicting earlier results.
\end{proof}

Thus, both maps $F \modu f_1^\star$ and $F \modu f_2^\star$ are local. By Lemma \ref{twovectorssymlemma}, the map $F$ is local.
This completes the inductive step. Hence, Theorem \ref{generalizedsymtheo} is now proved for all fields with more than $2$ elements.

\subsection{Proof of Theorem \ref{generalizedsymtheo} for fields with $2$ elements}

Here, we use a similar strategy as in the previous two proofs, but now we need a stronger version of Lemma \ref{existtwovectorssym}.
In its statement, we shall use the following notation:
$$\calT_3(\K):=\Biggl\{\begin{bmatrix}
a & b & c \\
b & d & 0 \\
c & 0 & e
\end{bmatrix} \mid (a,b,c,d,e) \in \K^5\Biggr\} \subset \Mats_3(\K).$$
Note that $\calT_3(\K)$ is a linear hyperplane of $\Mats_3(\K)$.

\begin{lemma}\label{existthreevectorssymF2}
Assume that $\# \K=2$.
Let $U$ be a finite-dimensional vector space, $V$ be a linear subspace of $U$, and $\calS$
be a linear subspace of $\calL_s(U,V^\star)$. Assume that $\codim \calS \leq \dim V-2$ and $\dim V \geq 3$.
Then:
\begin{enumerate}[(a)]
\item Either there exist distinct non-zero vectors $f_1^\star$, $f_2^\star$ and $f_3^\star$ in $V^\star$ such that
$\codim (\calS \modu f_i^\star) \leq \dim V-3$ for all $i \in \lcro 1,3\rcro$;
\item Or $\dim V=3$ and $\calS$ is represented, in a choice of bases of $U$ and $V^\star$, by
$\calT_3(\K) \coprod \Mat_{3,p-3}(\K)$, where $p:=\dim U$.
\end{enumerate}
\end{lemma}

\begin{proof}
We use a similar line of reasoning as in the proof of Lemma \ref{existtwovectorssym}.
Assume first that $\calS_r=\calL_s(V,V^\star)$.
If $\calS=\calL_s(U,V^\star)$, we simply pick
three distinct vectors in $V^\star \setminus \{0\}$.
If $\calS \subsetneq \calL_s(U,V^\star)$, then $\calS_m \subsetneq \calL(U/V,V^\star)$, and
hence there is no basis $(f_1^\star,\dots,f_n^\star)$ of $V^\star$ for which, for all $i \in\lcro 1,n\rcro$, the linear subspace $\calS_m$ contains every operator of $\calL(U/V,V^\star)$ with range $\K f_i^\star$; hence, there is a linear hyperplane $H$ of $V^\star$
such that, for all $f^\star \in V^\star \setminus H$, the space $\calS$ does not contain every operator of $\calL_s(U,V^\star)$
with range $\K f^\star$, leading to $\codim (\calS \modu f^\star) \leq \dim V-3$.
In any case, condition (a) is easily seen to hold.

In the rest of the proof, we assume that $\calS_r \subsetneq \calL_s(V,V^\star)$ and that condition (a) does not hold.
Then, there are distinct vectors $f_1^\star$ and $f_2^\star$ in $V^\star \setminus \{0\}$ such that, for all
$f^\star \in V^\star \setminus \{0,f_1^\star,f_2^\star\}$, one has $\codim (\calS \modu f^\star)>\dim V-3$.
For any such $f^\star$, we deduce that $\calS$ contains all the operators of $\calL_s(U,V^\star)$ with range $\K f^\star$.
Then, as those linear forms span $V^\star$, we learn that $\calS_m=\calL(U/V,V^\star)$. On the other hand,
as there is at least one such linear form we find $\codim \calS=\dim V-2$ and hence $\codim \calS_r=\dim V-2$.
As $f_1^\star$ and $f_2^\star$ are non-collinear (as they are distinct and non-zero, while $\# \K=2$), there is
a basis $\bfB^\star=(e_1^\star,\dots,e_n^\star)$ of $V^\star$ in which $f_1^\star=\underset{k=1}{\overset{n}{\sum}} e_k^\star$
and $f_2^\star=\underset{k=2}{\overset{n}{\sum}} e_k^\star$. Consider the predual basis $\bfB$ of $V$, and denote by $\calM$
the matrix space representing $\calS_r$ is the bases $\bfB$ and $\bfB^\star$.
Note that $\calM$ has codimension $n-2$ in $\Mats_n(\K)$.
Moreover, we know that, for every $X \in \K^n$ with the possible exception of $X_1:=\begin{bmatrix}
1 \\
\vdots \\
1
\end{bmatrix}$ and  $X_2:=\begin{bmatrix}
0 \\
1 \\
\vdots \\
1
\end{bmatrix}$, the space $\calM$ should contain $XX^T$. As in the proof of Lemma \ref{rank1lemma}, this shows:
\begin{itemize}
\item That $\calM$ contains $E_{i,i}$ for all $i \in \lcro 1,n\rcro$;
\item That $\calM$ contains $E_{i,j}+E_{i,j}$
for all $(i,j)\in \lcro 1,n\rcro^2$ such that $i \neq j$, unless $n=3$ in which case we can only assert that
$\calM$ contains $E_{1,2}+E_{2,1}$ and $E_{1,3}+E_{3,1}$.
\end{itemize}
If $n>3$, then we deduce that $\calM=\Mats_n(\K)$, which contradicts the fact that $\codim \calM=n-2$.
Therefore, $n=3$, and we see from the above that $\calT_3(\K) \subset \calM$. However $\dim \calT_3(\K)=5=\dim \calM$,
whence $\calM=\calT_3(\K)$.
Let us extend $\bfB$ into a basis $\widetilde{\bfB}$ of $U$.
Then, as $\calS_m=\calL(U/V,V^\star)$, we conclude that in the bases
$\widetilde{\bfB}$ and $\bfB^\star$ the matrix space $\calT_3(\K) \coprod \Mat_{3,p-3}(\K)$ represents $\calS$.
\end{proof}

Now, let us assume that $\# \K=2$. We prove Theorem \ref{generalizedsymtheo} by induction on the dimension of $V$.
Again, the case $\dim V \leq 2$ is either vacuous or trivial.
Assume now that $\dim V \geq 3$.

\begin{claim}\label{lastclaimsym}
There are distinct non-zero forms $f_1^\star$, $f_2^\star$, $f_3^\star$ in $V^\star$ such that
each map $F \modu f_i^\star$ is standard.
\end{claim}

\begin{proof}
If condition (a) from Lemma \ref{existthreevectorssymF2} holds, then the result is an obvious consequence of the induction hypothesis.
Assume now that $\calS$ is represented, in a compatible pair $(\bfB,\bfC^\star)$ of bases of $U$ and $V^\star$, by
$\calT_3(\K) \coprod \Mat_{3,p-3}(\K)$. Set $X_1:=\begin{bmatrix}
1 \\
1 \\
1
\end{bmatrix}$, $X_2:=\begin{bmatrix}
0 \\
1 \\
1
\end{bmatrix}$ and $X_3:=\begin{bmatrix}
1 \\
0 \\
0
\end{bmatrix}$ and note that $X_1,X_2,X_3$ are distinct non-zero vectors of $\K^3$.
Denote by $f_1^\star$, $f_2^\star$ and $f_3^\star$ the corresponding vectors of $V^\star$ (through $\bfB^\star$).
Then, as $\calT_3(\K)$ contains no rank $1$ matrix with range either $\K X_1$ or $\K X_2$, we see that
$\codim (\calS \modu f_1^\star)=\codim (\calS \modu f_2^\star)=0$.
On the other hand, we see that $\calS \modu f_3^\star$ is represented by the matrix space
$\K^2 \coprod D_1 \coprod D_2 \coprod \Mat_{2,p-3}(\K)$, where $D_1$ and $D_2$ respectively denote the $1$-dimensional linear subspaces
$\K \times \{0\}$ and $\{0\} \times \K$ of $\K^2$.
Combining Theorem \ref{totalspace} with the splitting lemma, we find that every range-compatible linear map on $\calS \modu f_3^\star$
is local. Thus, as here $\#\K=2$, every range-compatible homomorphism on $\calS \modu f_3^\star$
is local. In particular, $F \modu f_i^\star$ is standard for all $i \in \{1,2,3\}$.
\end{proof}

We are now about to conclude the proof.
Let us pick distinct non-zero forms $f_1^\star,f_2^\star,f_3^\star$ satisfying the conclusion of Claim \ref{lastclaimsym}.
Without loss of generality, we can assume that $F \modu f_1^\star$ and $F \modu f_2^\star$
are both local or both non-local.
In the first case, Lemma \ref{twovectorssymlemma} yields that $F$ is local.
Assume now that $F \modu f_1^\star$ and $F \modu f_2^\star$ are both non-local.
We know that $F \modu f_1^\star=G_1 \modu f_1^\star$ and $F \modu f_2^\star=G_2 \modu f_2^\star$
for some standard range-compatible homomorphisms $G_1$ and $G_2$ on $\calS$.
Neither $G_1$ nor $G_2$ is local. Moreover, the identity is the sole non-zero root-linear map from $\K$ to $\K$.
Hence, from the general form of standard maps that is given in Proposition \ref{operatorsymprop},
we deduce that $G_1-G_2$ is local. Thus, replacing $F$ with $F'=F-G_1$, we see that both maps
$F' \modu f_1^\star$ and $F' \modu f_2^\star$ are local. It follows again from Lemma \ref{twovectorssymlemma}
that $F'$ is local, and we conclude that $F=F'+G_1$ is standard.

Thus, the inductive step is completed for fields with two elements. Theorem \ref{generalizedsymtheo}
is now established, and Theorem \ref{maintheosym} follows from it.

\section{Range-compatible homomorphisms on large spaces of alternating matrices}\label{RCalt}

\subsection{Main results}

As in the symmetric case, it is necessary, in order to obtain Theorem \ref{alttheo}, to consider a general version that applies to a larger class
of operator spaces.

\begin{Def}
Let $U$ be a finite-dimensional vector space and $V$ be a linear subspace of $U$.
A linear map $f : U \rightarrow V^\star$ is called \textbf{alternating} when
$$\forall x \in V, \; f(x)[x]=0,$$
i.e.\ when the bilinear form $(x,y) \in V^2 \mapsto f(x)[y]$ is alternating.
The set of all alternating linear maps from $U$ to $V^\star$ is denoted by $\calL_a(U,V^\star)$.
One checks that it is a linear subspace of $\calL(U,V^\star)$.
\end{Def}

For every compatible pair $(\bfB_1,\bfB_2^\star)$ of bases of $U$ and $V^\star$,
a map from $U$ to $V^\star$ is alternating if and only if its matrix in $\bfB_1$ and $\bfB_2^\star$
belongs to $\Mata_n(\K) \coprod \Mat_{n,p-n}(\K)$, where $n:=\dim V$ and $p:=\dim U$.

\begin{Rem}
Let $W$ be a linear subspace of $V^\star$.
Through the canonical identification between $V^\star/W$ and $({}^o W)^\star$,
the space $\calL_a(U,V^\star) \modu W$ corresponds to $\calL_a(U,({}^o W)^\star)$.

In particular, $\calL_a(U,V^\star)$ corresponds to $\calL_a(U,U^\star) \modu V^o$.

As in Remark \ref{Rem4}, such identifications have no impact on the specific problem we are studying.
\end{Rem}

Let $\calS$ be a linear subspace of $\calL_a(U,V^\star)$.
To $\calS$, we naturally attach two operator spaces:
the space
$$\calS_r:=\bigl\{f_{|V} \mid f \in \calS\bigr\},$$
which is a linear subspace of $\calL_a(V,V^\star)$,
and the space of all $f \in \calS$ such that $f$ vanishes everywhere on $V$, which we denote by $\calS_m$
 and naturally identify with a linear subspace of $\calL(U/V, V^\star)$.
By the rank theorem, we have
$$\dim \calS=\dim \calS_r+\dim \calS_m$$
whereas
$$\dim \calL_a(U,V^\star)=\dim \calL_a(V,V^\star)+\dim \calL(U/V,V^\star).$$
Let $(\bfB_1,\bfB_2^\star)$ be a compatible pair of bases of $U$ and $V^\star$, and denote by $\calM$ the matrix space representing
$\calS$ in them. Write $\bfB_1=(e_1,\dots,e_p)$ and set $\bfB'_1:=(e_1,\dots,e_n)$ and $\bfB''_1:=(\overline{e_{n+1}},\dots,\overline{e_p})$,
which are bases of $V$ and $U/V$, respectively.
Every $M \in \calM$ splits as
$$M=\begin{bmatrix}
S(M) & R(M)
\end{bmatrix} \quad \text{with $S(M) \in \Mata_n(\K)$ and $R(M) \in \Mat_{n,p-n}(\K)$.}$$
Then, the matrix space $S(\calM)$ represents $\calS_r$ in the bases $\bfB'_1$ and $\bfB_2^\star$.
On the other hand, if we denote by $\calT$ the subspace of $\calS$ consisting of all its matrices $M$ such that $S(M)=0$,
then $R(\calT)$ represents $\calS_m$ in the bases $\bfB''_1$ and $\bfB_2^\star$.

The proper generalization of Theorem \ref{alttheo} to spaces of alternating linear operators is the following one.

\begin{theo}\label{generalizedalttheo}
Let $U$ be a finite-dimensional vector space and $V$ be a linear subspace of $U$.
Let $\calS$ be a linear subspace of $\calL_a(U,V^\star)$ and assume that\footnote{
Here, $\codim \calS$ and $\codim \calS_r$ stand, respectively, for the codimension of $\calS$ in $\calL_a(U,V^\star)$ and
for the one of $\calS_r$ in $\calL_a(V,V^\star)$.}
$\codim \calS \leq \dim V-2$ and $\codim \calS_r \leq \dim V-3$.
Then, every range-compatible homomorphism on $\calS$ is local.
\end{theo}

Again, the proof of Theorem \ref{generalizedalttheo} will be done by induction on the dimension of $V$.
The case when $\dim V=3$ is dealt with in Section \ref{altdim3}, and the inductive step
is performed in Section \ref{altinduction}.

\subsection{The case $\dim V=3$}\label{altdim3}

Here, we let $U$ be a finite-dimensional vector space and $V$ be a linear subspace of $U$ with dimension $3$.
Let $\calS$ be a linear subspace of $\calL_a(U,V^\star)$ such that $\codim \calS_r \leq \dim V-3$
and $\codim \calS \leq \dim V-2$. Then, $\calS_r=\calL_a(V,V^\star)$ and
$\calS_m$ has codimension at most $1$ in $\calL(U/V,V^\star)$. Set $p:=\dim U$.

\begin{Not}
For any non-negative integer $r$, we denote by $\mathfrak{sl}_r(\K)$ the space of all trace zero matrices in $\Mat_r(\K)$.
\end{Not}

Let us consider the bilinear form
$$b : (s,t) \in \calL(U/V,V^\star) \times \calL(V^\star,U/V) \mapsto \tr(v \circ u).$$
The orthogonal complement of $\calS_m$ with respect to $b$ is then a linear subspace $\calS_m^\bot$ of $\calL(V^\star,U/V)$
with dimension at most $1$, and its orthogonal complement for $b$ is $\calS_m$.
There are two main cases to consider:
\begin{itemize}
\item \textbf{Case 1: $\calS_m^\bot=\{0\}$.} \\
Then, $\calS=\calL_a(U,V^\star)$.
\item \textbf{Case 2: $\calS_m^\bot$ contains a non-zero operator, say $t$, with rank denoted by $r$.} \\
We can find bases $(e_1^\star,e_2^\star,e_3^\star)$ and $(\overline{f_4},\dots,\overline{f_{p}})$ of $V^\star$ and $U/V$
(with $f_4,\dots,f_{p}$ in $U$) in which
$t$ is represented by $\begin{bmatrix}
I_r & [0]_{r \times (3-r)} \\
[0]_{(p-3-r) \times r} & [0]_{(p-3-r) \times (3-r)}
\end{bmatrix}$. Then, in those bases, the space $\calS_m$ is represented by the set of all matrices of the form
$\begin{bmatrix}
N & [?]_{r \times (p-3-r)} \\
[?]_{(3-r) \times r} & [?]_{(3-r) \times (p-3-r)}
\end{bmatrix}$ with $N \in \mathfrak{sl}_r(\K)$.
If we denote by $(e_1,e_2,e_3)$ the predual basis of $(e_1^\star,e_2^\star,e_3^\star)$,
we find that $\bfB=(e_1,e_2,e_3,f_4,\dots,f_p)$ is a basis of $U$.
Finally, as $\calS_r=\calL_a(V,V^\star)$, we conclude that there is a linear map $f : \Mata_3(\K) \rightarrow \Mat_r(\K)$
such that $\calS$ is represented by the space of all matrices of the form
$$\begin{bmatrix}
A & [0]_{3 \times (p-3)}
\end{bmatrix}+
\begin{bmatrix}
[0]_{r \times 3} & N+f(A) & [?]_{r \times (p-3-r)} \\
[0]_{(3-r) \times 3} & [?]_{(3-r) \times r} & [?]_{(3-r) \times (p-3-r)}
\end{bmatrix}$$
with $A \in \Mata_3(\K)$ and $N \in \mathfrak{sl}_r(\K)$.
\end{itemize}

Case 2 helps motivate the following notation.

\begin{Not}
Let $r \in \lcro 0,3\rcro$ and $f : \Mata_3(\K) \rightarrow \Mat_r(\K)$ be a linear map.
We denote by $\calM_f$ the space of all $3$ by $3+r$ matrices of the form
$$\begin{bmatrix}
A & [0]_{3 \times r}
\end{bmatrix}+
\begin{bmatrix}
[0]_{r \times 3} & N+f(A) \\
[0]_{(3-r) \times 3} & [?]_{(3-r) \times r} \\
\end{bmatrix}$$
with $A \in \Mata_3(\K)$ and $N \in \mathfrak{sl}_r(\K)$.
\end{Not}

Note that $\calM_f=\Mata_3(\K)$ if $r=0$.
To sum up, we have proved that there exists an integer $r \in \lcro 0,3\rcro$, a linear map $f : \Mata_3(\K) \rightarrow \Mat_r(\K)$
and a compatible pair of bases in which $\calS$ is represented by the matrix space
$\calM_f \coprod \Mat_{3,p-3-r}(\K)$.
However, we know from Theorem \ref{maintheogroup} that every range-compatible homomorphism on
$\Mat_{3,p-3-r}(\K)$ is local. Thus, the splitting lemma shows that in order to conclude, it remains to prove the following lemma.

\begin{lemma}\label{dim3lemma}
Let $r \in \lcro 0,3\rcro$ and $f : \Mata_3(\K) \rightarrow \Mat_r(\K)$ be a linear map.
Then, every range-compatible homomorphism on $\calM_f$ is local.
\end{lemma}

To prove Lemma \ref{dim3lemma}, we start with the case $r=0$.

\begin{lemma}\label{dim3lemmar=0}
Every range-compatible homomorphism on $\Mata_3(\K)$ is local.
\end{lemma}

\begin{proof}[Proof of Lemma \ref{dim3lemmar=0}]
Let $F : \Mata_3(\K) \rightarrow \K^3$ be a range-compatible homomorphism.
By Corollary \ref{rowlemma}, there are group endomorphisms $u,g,h,i,j,k$ of $\K$ such that
$$F : \begin{bmatrix}
0 & -a & b \\
a & 0 & -c \\
-b & c & 0
\end{bmatrix} \mapsto \begin{bmatrix}
-u(a)+g(b) \\
h(a)-i(c) \\
-j(b)+k(c)
\end{bmatrix}.$$
For all $(a,b,c)\in \K^3 \setminus \{(0,0,0)\}$, the range of the alternating matrix $\begin{bmatrix}
0 & -a & b \\
a & 0 & -c \\
-b & c & 0
\end{bmatrix}$ is the orthogonal complement of its kernel, which is spanned by $\begin{bmatrix}
c \\
b \\
a
\end{bmatrix}$. This leads to
\begin{equation}\label{identitealt}
\forall (a,b,c)\in \K^3, \; -u(a)c+g(b)c+h(a)b-i(c)b-j(b)a+k(c)a=0.
\end{equation}
Fixing $b=0$ and $c=1$ and varying $a$, we obtain that $u$ is linear. Similarly, one shows that
$g,h,i,j,k$ are all linear.
Then, \eqref{identitealt} reads
$$\forall (a,b,c)\in \K^3, \; \bigl(-u(1)+k(1)\bigr)ac+\bigl(g(1)-i(1)\bigr)bc+\bigl(h(1)-j(1)\bigr)ab=0.$$
On the left-hand side of this identity, we have a polynomial of degree at most $1$ in each variable, and hence all its coefficients are zero.
This leads to $u(1)=k(1)$, $g(1)=i(1)$ and $h(1)=j(1)$.
Therefore, $F$ is the local map $A \mapsto AX$ where $X:=\begin{bmatrix}
h(1) \\
u(1) \\
g(1)
\end{bmatrix}$.
\end{proof}

Now, we prove Lemma \ref{dim3lemma}.

\begin{proof}[Proof of Lemma \ref{dim3lemma}]
The case $r=0$ has already been dealt with in Lemma \ref{dim3lemmar=0}.

Assume now that $r \geq 1$. Let $F : \calM_f \rightarrow \K^3$ be a range-compatible homomorphism. \\
Denote by $\calT$ the subspace of $\calM_f$ consisting of its matrices
with all first three columns equal to $0$. Then, we see that $\calT=\{0\} \coprod \calN$ for some
subspace $\calN$ of $\Mat_{3,r}(\K)$ with codimension $1$.
By the splitting lemma and Theorem \ref{maintheogroup}, the restriction of $F$ to $\calT$ must be local.
Thus, subtracting a well-chosen local map from $F$, no generality is lost in assuming that $F$ vanishes everywhere on $\calT$,
a condition that will be assumed to hold throughout the rest of the proof.
It follows that there is a homomorphism
$$G : \Mata_3(\K) \rightarrow \K^3$$
such that, for all $M \in \calM_f$, we have $F(M)=G(K(M))$, where
$$M=\begin{bmatrix}
K(M) & [?]_{3 \times r}
\end{bmatrix} \quad \text{with $K(M) \in \Mata_3(\K)$.}$$
We shall prove that $G$ is range-compatible.
To do so, we distinguish between three cases.

\vskip 3mm
\noindent \textbf{Case a: $r=3$.} \\
Let $A \in \Mata_3(\K) \setminus \{0\}$.
We choose a non-zero column $C$ of $A$.
The subspace $\{C X^T \mid X \in \K^3\}$ of $\Mat_3(\K)$
is not included in $\mathfrak{sl}_3(\K)$, and hence we can find a rank $1$ matrix $B \in \Mat_3(\K) \setminus \mathfrak{sl}_3(\K)$
whose range equals $\K C$. Then, $\Mat_3(\K)=\K B \oplus \mathfrak{sl}_3(\K)$,
and hence we can write $f(A)=\lambda\,B+N$ for some $\lambda \in \K$ and some $N \in \mathfrak{sl}_3(\K)$.
Thus, $\calM_f$ contains $M:=\begin{bmatrix}
A & \lambda B
\end{bmatrix}$, whose range equals $\im A$. Then, $G(A)=F(M) \in \im M$, and hence $G(A) \in \im A$.

\vskip 3mm
\noindent \textbf{Case b: $r=2$.} \\
Let $A \in \Mata_3(\K) \setminus \{0\}$.
We prove that there exists a matrix $M \in \calM_f$ such that
$M=\begin{bmatrix}
A & [?]_{3 \times 2}
\end{bmatrix}$ and $\im M=\im A$, which, as in Case a, will entail that $G(A) \in \im A$.

One of the columns of $A$, say the $i$-th, can be written as $\begin{bmatrix}
Y \\
?
\end{bmatrix}$ with $Y \in \K^2 \setminus \{0\}$.
As in Case a, we find $B \in \Mat_2(\K) \setminus \mathfrak{sl}_2(\K)$ with range $\K Y$.
Then, $f(A)=\lambda\,B+N$ for some $\lambda \in \K$ and some $N \in \mathfrak{sl}_2(\K)$.
It follows that we can find a row matrix $L \in \Mat_{1,2}(\K)$ such that
$\begin{bmatrix}
\lambda\,B \\
L
\end{bmatrix}$ has all its columns collinear to the $i$-th column of $A$.
Then, the matrix
$$\begin{bmatrix}
A & [0]_{3 \times 2}
\end{bmatrix}+\begin{bmatrix}
[0]_{2 \times 3} & \lambda\,B \\
[0]_{1 \times 3} & L
\end{bmatrix}$$
belongs to $\calM_f$ and its range equals $\im A$, which proves our claim.

\vskip 3mm
\noindent \textbf{Case c: $r=1$.} \\
Using Corollary \ref{rowlemma}, we find
endomorphisms $u,g,h,i,j,k,\theta$ of $(\K,+)$ such that
$$F : \begin{bmatrix}
0 & -a & b & ? \\
a & 0 & -c & d \\
-b & c & 0 & e
\end{bmatrix} \mapsto \begin{bmatrix}
-u(a)+g(b)+\theta(c) \\
h(a)-i(c) \\
-j(b)+k(c)
\end{bmatrix}.$$
Let $A \in \Mata_3(\K)$ be with non-zero first row.
Again, we see that there exists some $M \in \calM_f$ of the form $M=\begin{bmatrix}
A & [?]_{3 \times 1}
\end{bmatrix}$ such that $\im M=\im A$: indeed, the $2$-dimensional affine subspaces
$\begin{bmatrix}
f(A) \\
0 \\
0
\end{bmatrix}+(\{0\} \times \K^2)$ and $\im A$ of $\K^3$ are non-parallel, and hence
they have a common point. Thus,
$G(A)=F(M) \in \im A$.

As in the proof of Lemma \ref{dim3lemmar=0}, this yields
$$\forall (a,b,c)\in \K^3, \; (a,b) \neq (0,0) \Rightarrow
-u(a)c+g(b)c+h(a)b-i(c)b-j(b)a+k(c)a+\theta(c)c=0.$$
Next, we prove that $\theta=0$. We distinguish between two cases:
\begin{itemize}
\item Assume that $\# \K>2$. Let $c \in \K \setminus \{0\}$.
Taking $b=0$, we find that $-u(a)c+k(c)a+\theta(c)c=0$ for all $a \in \K \setminus \{0\}$.
Thus, the group endomorphism $\chi : a \in \K \mapsto -u(a)c+k(c)a$ is constant on $\K \setminus \{0\}$
with sole value $-\theta(c)c$. However, since $\K$ has more than $2$ elements, we can find
$x$ and $y$ in $\K \setminus \{0\}$ such that $x+y \neq 0$, which leads to $\chi(x)+\chi(y)=\chi(x+y)$, and hence
$\theta(c)c=0$. Therefore, $\theta(c)=0$.

\item Assume that $\# \K=2$. Then, $u,g,h,i,j,k,\theta$ are all linear, and hence the quadratic form
$$(a,b,c) \mapsto (-u(1)+k(1))ac+(g(1)-i(1))bc+(h(1)-j(1))ab+\theta(1) c^2=0$$
vanishes everywhere on $\K^3 \setminus \Vect\bigl((0,0,1)\bigr)$.
By Lemma 5.2 of \cite{dSPRC1}, $q=0$, and in particular $\theta(1)=0$.
\end{itemize}
Therefore, in any case $\theta=0$. Moreover, if $(a,b)=(0,0)$ it is obvious that
$$-u(a)c+g(b)c+h(a)b-i(c)b-j(b)a+k(c)a=0.$$
Hence,
$$\forall (a,b,c)\in \K^3, \; -u(a)c+g(b)c+h(a)b-i(c)b-j(b)a+k(c)a=0,$$
which shows that $G$ is range-compatible (see the proof of Lemma \ref{dim3lemmar=0}).

\vskip 3mm
Thus, in any case we have shown that $G$ is range-compatible. Lemma \ref{dim3lemmar=0} then yields
a vector $Y \in \K^3$ such that $G : A \mapsto AY$, and we conclude that $F$ is the local map $M \mapsto MX$ for
$X:=\begin{bmatrix}
Y \\
[0]_{r \times 1}
\end{bmatrix}$.
\end{proof}

This completes the proof in the case when $\dim V=3$.

\subsection{The case $\dim V>3$}\label{altinduction}

Here, we perform the inductive step of the proof of Theorem \ref{generalizedalttheo}. Set $n:=\dim V$, $p:=\dim U$ and assume that $\dim V>3$.

\begin{step}
There exist non-collinear forms $f_1^\star$ and $f_2^\star$ in $V^\star$ such that
$\codim (\calS \modu f_i^\star) \leq n-3$ and $\codim (\calS \modu f_i^\star)_r \leq n-4$ for all
$i \in \{1,2\}$.
\end{step}

\begin{proof}
Let $f^\star \in V^\star \setminus \{0\}$.
Denote by $\calS_r^{f^\star}$ the subspace of all the operators $s \in \calS_r$ such that
$(x,y) \mapsto s(x)[y]$ vanishes everywhere on $\Ker f^\star \times \Ker f^\star$.
Then, by the rank theorem, we have
$$\codim (\calS \modu f^\star)_r =\codim \calS_r +\bigl(\dim \calS_r^{f^\star}-\dim V+1\bigr).$$
Thus, in order to have
$$\codim (\calS \modu f^\star)_r \leq n-4,$$
it suffices that $\dim \calS_r^{f^\star} < n-1$.
Assume that there are $n-1$ linearly independent vectors $f_1^\star,\dots,f_{n-1}^\star$ in $V^\star$
such that $\dim \calS_r^{f_i^\star}=\dim V-1$ for all $i \in \lcro 1,n\rcro$.
Then, we extend this family into a basis $\bfB^\star=(f_1^\star,\dots,f_{n-1}^\star,f_n^\star)$ of $V^\star$, whose corresponding
basis of $V$ is denoted by $\bfB$. Denote by $\calM$ the matrix space associated with $\calS_r$ in the bases
$\bfB$ and $\bfB^\star$. The assumptions show that, for all $(i,j)\in \lcro 1,n\rcro^2$ such that $i>j$, the space $\calM$
contains $E_{i,j}-E_{j,i}$ (since $\dim \calS_r^{f_j^\star}=\dim V-1$).
By linearly combining those matrices, we find that $\calM=\Mata_n(\K)$, which means that
$\calS_r=\calL_a(V,V^\star)$. In turn, this shows that
for all $f^\star \in V^\star \setminus \{0\}$, we have $(\calS \modu f^\star)_r=\calL_a\bigl(\Ker f^\star,(\Ker f^\star)^\star\bigr)$ and hence
$$\codim (\calS \modu f^\star)_r=0 \leq n-4.$$
Thus, in any case, we have found a linear subspace $P$ of codimension $2$ in $V^\star$
such that
$$\forall f^\star \in V^\star \setminus P, \; \codim (\calS \modu f^\star)_r \leq n-4.$$

Let $f \in V^\star \setminus \{0\}$.
If we denote by $\calS_m^{f^\star}$ the space of all linear maps
$s \in \calS_m$ such that $\im s \subset \K f^\star$, the rank theorem yields
$$\codim (\calS \modu f^\star) = \codim \calS +\bigl(\dim(\calS_m^{f^\star})-\dim (U/V)\bigr).$$
Indeed, the rank of an alternating linear map from $V$ to $V^\star$
cannot equal $1$, and hence if $s \in \calS$ satisfies $\im s \subset \K f^\star$ then $s$ vanishes everywhere on $V$,
to the effect that $s \in \calS_m$.

Thus, in order to have
$$\codim (\calS \modu f^\star) \leq n-3,$$
it suffices that $\dim \calS_m^{f^\star} <\dim (U/V)$.
If there existed a basis $(f_1^\star,\dots,f_n^\star)$ of $V^\star$ such that
$\dim \calS_m^{f_i^\star}=\dim (U/V)$ for all $i \in \lcro 1,n\rcro$, then it would follow that
$\calS_m=\calL(U/V,V^\star)$, and hence
$$\codim \calS=\codim \calS_r \leq n-3.$$
In that case, we would obtain
$$\codim (\calS \modu f^\star) \leq \codim \calS \leq n-3$$
whatever the choice of $f^\star$ in $V^\star \setminus \{0\}$.

This yields a linear hyperplane $H$ of $V^\star$ such that
$$\forall f^\star \in V^\star \setminus H, \; \codim (\calS \modu f^\star) \leq n-3.$$
Hence, for every $f^\star \in V^\star \setminus (P \cup H)$, the space $\calS \modu f^\star$ satisfies the conditions of Theorem \ref{generalizedalttheo}.
However, the union of two proper linear subspaces of $V^\star$ is always a proper subset of $V^\star$:
this successively yields some $f_1^\star \in V \setminus (P \cup H)$, and then we can pick $f_2^\star$ in $V \setminus ((P+\K f_1^\star) \cup H)$,
so that $f_1^\star$ and $f_2^\star$ are linearly independent and both spaces
$\calS \modu f_1^\star$ and $\calS \modu f_2^\star$ satisfy the requirements of Theorem \ref{generalizedalttheo}.
\end{proof}

By induction, we obtain that both range-compatible homomorphisms $F \modu f_1^\star$ and $F \modu f_2^\star$
are local, yielding vectors $x_1$ and $x_2$ in $U$ such that
$$\forall s \in \calS, \quad F(s)=s(x_1) \mod \K f_1^\star \quad \text{and} \quad F(s)=s(x_2) \mod \K f_2^\star.$$
Replacing $F$ with $s \mapsto F(s)-s(x_1)$, we reduce the situation further to the point where there exists $x \in U$ such that
\begin{equation}\label{localcond}
\forall s \in \calS, \; F(s) \in \K f_1^\star \quad \text{and} \quad F(s)=s(x) \mod \K f_2^\star.
\end{equation}
Notice then that
$$\forall s \in \calS, \quad s(x) \in \Vect(f_1^\star,f_2^\star).$$

\begin{step}
The vector $x$ belongs to $V$.
\end{step}

\begin{proof}
Assume on the contrary that $x\not\in V$.
Let us extend $(f_1^\star,f_2^\star)$ into a basis
$\bfB^\star$ of $V^\star$, whose predual basis of $V$ is denoted by $\bfB=(e_1,\dots,e_n)$, and
let us extend $(e_1,\dots,e_n,x)$ into a basis $\widetilde{\bfB}$ of $U$.
Denote by $\calM$ the matrix space representing $\calS$ in $\widetilde{\bfB}$ and $\bfB^\star$.
Then, if we denote by $P$ the subspace $\K^2 \times \{0\}$ of $\K^n$, we see that
$\calM$ is a linear subspace of $\Mata_n(\K) \coprod P \coprod \Mat_{n,p-n-1}(\K)$, where $p:=\dim U$.
However, that subspace has codimension $n-2$ in $\Mata_n(\K) \coprod \Mat_{n,p-n}(\K)$, whereas
$\calM$ has codimension at most $n-2$ in it. Therefore,
$\calM=\Mata_n(\K) \coprod P \coprod \Mat_{n,p-n-1}(\K)$.

Denote by $G : \calM \rightarrow \K^n$ the range-compatible homomorphism that is associated with $F$ in the above bases.
Condition \eqref{localcond} shows that $G$ maps every matrix $M=(m_{i,j})$ of $\calM$ to $\begin{bmatrix}
m_{1,n+1}\\
[0]_{(n-1) \times 1}
\end{bmatrix}$.
Then, we have a range-compatible homomorphism $H : P \rightarrow \K^n$ such that, for all $Y \in P$,
$$G\Bigl(\begin{bmatrix}
[0]_{n \times n} & Y & [0]_{n \times (p-n-1)}
\end{bmatrix}\Bigr) =H(Y),$$
and we see from the above form of $G$ that $\im H=\K \times \{0\}$. However, by Theorem \ref{maintheogroup}
the map $H$ should be a scalar multiple of the identity, which is false since $P$ has dimension $2$.
\end{proof}

\begin{step}
The vector $x$ belongs to $\Ker f_1^\star \cap \Ker f_2^\star$.
\end{step}

\begin{proof}
Assume on the contrary that $x \not\in {}^o \Vect(f_1^\star,f_2^\star)$.
Then, we can find a basis $(g_1^\star,g_2^\star)$ of $\Vect(f_1^\star,f_2^\star)$ such that $g_1^\star(x)=1$ and $g_2^\star(x)=0$.
We extend this basis into a basis $\bfB^\star=(g_1^\star,g_2^\star,\dots,g_n^\star)$ of $V^\star$ in which $g_i^\star(x)=0$ for all $i \in \lcro 2,n\rcro$,
and we consider the corresponding basis $\bfB$ of $V$.
Obviously, $x$ is the first vector of $\bfB$. It then turns out that if we denote by $\calM_r$ the matrix space associated with
$\calS_r$ in the bases $\bfB$ and $\bfB^\star$, this space is included in $\Mata_n(\K)$
and every matrix in $\calM_r$ has its first column of the form
$\begin{bmatrix}
0 \\
? \\
[0]_{(n-2) \times 1}
\end{bmatrix}$. Thus, we see that $\calM_r$ has codimension at least $n-2$ in $\Mata_n(\K)$,
contradicting the assumption that $\codim \calS_r \leq n-3$.
\end{proof}

\begin{step}
One has $x=0$.
\end{step}

\begin{proof}
Assume on the contrary that $x \neq 0$.
Let us consider a linear form $f_0^\star$ on $V$ such that $f_0(x)=1$.
Then, we extend $(f_0^\star,f_1^\star,f_2^\star)$ into a basis $\bfB^*$ in which the last $n-3$ linear forms
annihilate $x$, and we denote by $\bfB$ the corresponding basis of $V$, which we extend into a basis $\widetilde{\bfB}$ of $U$.
Note that $x$ is the first vector of $\bfB$, and hence also the first one of  $\widetilde{\bfB}$.
Let us denote by $\calM$ the matrix space associated with $\calS$ in those bases.
For $M \in \calM$ we write
$$M=\begin{bmatrix}
A(M) & B(M)
\end{bmatrix} \quad \text{with $A(M)\in \Mata_n(\K)$ and $B(M) \in \Mat_{n,p-n}(\K)$.}$$
Then, we know that the first column of every matrix of $\calM$
has the form $\begin{bmatrix}
0 \\
? \\
? \\
[0]_{(n-3) \times 1}
\end{bmatrix}$.

Denote by $G : \calM \rightarrow \K^n$ the range-compatible homomorphism that corresponds to $F$ in the above bases.
Then, $G$ maps every matrix $(m_{i,j})$ of $\calM$ to the vector $\begin{bmatrix}
0 \\
m_{2,1} \\
[0]_{(n-2) \times 1}
\end{bmatrix}$.

Denote by $\calN$ the linear subspace of $\Mata_n(\K)$ consisting of its matrices $(m_{i,j})$ such that
$m_{i,1}=0$ for all $i \in \lcro 4,n\rcro$.
Then, $\codim \calN=n-3$. However, $\codim A(\calM)=\codim \calS_r \leq n-3$, whence
$A(\calM)=\calN$ and $\codim \calS_r=n-3$.
It follows that we can find a linear mapping
$$K : \Mata_3(\K) \longrightarrow \Mat_{3,p-n}(\K)$$
such that, for all $C \in \Mata_3(\K)$, the space $\calM$ contains a matrix of the form
$$\begin{bmatrix}
C & [0]_{3 \times (n-3)} & K(C) \\
[0]_{(n-3) \times 3} & [0]_{(n-3) \times (n-3)} & [?]_{(n-3) \times (p-n)}
\end{bmatrix}.$$
Now, denote by $\calH$ the subspace of all matrices $M \in \calM$ such that $A(M)=0$.
Then $\codim B(\calH)+\codim \calS_r=\codim \calS$, whence $\codim B(\calH) \leq 1$.
Thus, $B(\calH)$ includes a linear hyperplane of $\Mat_{n,p-n}(\K)$.

For $M \in \calH$, denote by $B'(M)$ the sub-matrix of $B(M)$ obtained by deleting all the rows starting from the fourth one,
so that $B'(\calH)$ is a linear subspace of $\Mat_{3,p-n}(\K)$ with codimension at most $1$.
Thus, for all $C \in \Mata_3(\K)$ and $N' \in B'(\calH)$, the space
$\calM$ contains a matrix of the form
$$\begin{bmatrix}
C & [0]_{3 \times (n-3)} & K(C)+N' \\
[0]_{(n-3) \times 3} & [0]_{(n-3) \times (n-3)} & [?]_{(n-3) \times (p-n)}
\end{bmatrix}.$$
Let us finally denote by $\calM'$ the space of all matrices
$$\begin{bmatrix}
C & K(C)+N'
\end{bmatrix} \quad \text{with $C \in \Mata_3(\K)$ and $N' \in B'(\calH)$.}$$

Applying Lemma \ref{quotientgeneral} to the subspace $V_0=\{0\} \times \K^{n-3}$ of $\K^n$,
we obtain that
$$G' : (m_{i,j}) \in \calM' \mapsto \begin{bmatrix}
0 \\
m_{2,1} \\
0
\end{bmatrix}$$
is a range-compatible homomorphism. On the other hand, $\calM'$ represents a
space of alternating operators that satisfies the assumptions of Theorem \ref{generalizedalttheo}.
From the $3$-dimensional case treated in Section \ref{altdim3}, we deduce that $G'$ is a local map.
This would yield a vector $X \in \K^{3+p-n}$ such that $\dim \calM' X =1$.
However, if $X \in \K^3 \times \{0\}$ we would have $\dim \calM' X=2$ or $\dim \calM' X=0$;
on the other hand, if $X \not\in (\K^3 \times \{0\})$, we see that $\dim \calM' X \geq 2$
because $B'(\calH)$ has codimension at most $1$ in $\Mat_{3,p-n}(\K)$.
Thus, we have a contradiction in any case, and we conclude that $x=0$.
\end{proof}

We conclude that, for all $s \in \calS$, the vector $F(s)$ belongs to both $\K f_1^\star$ and $\K f_2^\star$,
and hence it equals $0$. Thus, $F$ is the local map $s \mapsto s(0)$, which completes the inductive step.
Thus, Theorem \ref{generalizedalttheo} is now established, and Theorem \ref{alttheo} follows as it is
a reformulation of the special case when $U=V$.

\end{document}